\documentclass[a4paper,11pt]{article}

\usepackage{float}
\usepackage{amsmath}
\usepackage{amsthm}
\usepackage{amsfonts}
\usepackage{amssymb}
\usepackage[T1]{fontenc}
\usepackage{geometry}
\usepackage{graphicx}
\usepackage{hyperref}
\usepackage[latin1]{inputenc}
\usepackage{latexsym}
\usepackage{textcomp}
\usepackage[all]{xy}
\usepackage{color}

\title{Uniform tightness for time-inhomogeneous particle systems and for conditional distributions of time-inhomogeneous diffusion processes}
\author{Denis Villemonais
\thanks{TOSCA project-team, INRIA Nancy -- Grand Est, IECN -- UMR 7502, Universit\'e Henri Poincar\'e Nancy 1
B.P. 70239, 54506 Vandoeuvre-l\`es-Nancy Cedex, France }
}

\DeclareMathSymbol{\minus}{\mathord}{operators}{"2D}

\newtheorem{theorem}{Theorem}
\newtheorem{lemme}[theorem]{Lemma}

\newtheorem{hypothesis}{Hypothesis}

\def\be{\begin{eqnarray}}
\def\ee{\end{eqnarray}}
\def\ben{\begin{eqnarray*}}
\def\een{\end{eqnarray*}}
\def\bei{\begin{itemize}}
\def\eei{\end{itemize}}
\def\me{\medskip \noindent}
\def\bi{\bigskip \noindent}
\def\E{\mathbb{E}}
\def\P{\mathbb{P}}
\def\R{\mathbb{R}}
\def\1{\mathbf{1}}

\def\d{\partial}

\begin{document}
\maketitle
\begin{abstract}
In this article, we consider time-inhomogeneous diffusive particle
systems, whose particles jump from the boundary of a bounded open
subset of $\R^d$, $d\geq 1$. We give a sufficient criterion for the
family of empirical distributions of such systems to be uniformly
tight, independently of the jump location of the particles. As an
application, we show that the conditional distribution of a family of
time-inhomogeneous and environment-dependent diffusions conditioned
not to hit the boundary of a bounded open subset of $\R^d$ is
uniformly tight.
\end{abstract}


\emph{MSC 2000 subject :} Primary 82C22, 65C50, 60K35; Secondary: 60J60

\section{Introduction}

Particle systems have received a lot of interest in the past decades
and have proved to be useful in several domains, as rare
events simulation and conditional distribution approximations.
In this paper, we consider a large class of particle systems
with fixed size $N\geq 2$, whose particles evolve as independent
diffusion processes in a given bounded open subset $D$ of $\mathbb{R}^d$
and jump when they hit the boundary $\partial D$ or after some exponential clocks. The distribution of the jump location depends on the whole system of particles, so that the particles are interacting with each others. The particles of such system evolve in the open set $D$, thus their empirical distribution is a random probability measure on $D$. In this paper, we are concerned with the uniform tightness of the law of the empirical distributions of particle systems with jump from a boundary.

\me
Such general interacting particle systems with jump from the boundary are directly inspired by the Fleming-Viot type system introduced by Burdzy, Holyst, Ingerman and March~\cite{Burdzy1996} in 1996. In
this paper, the authors describe a particle system whose particles
evolve as independent Brownian motions until one of them reaches the
boundary of a fixed open subset $D$ of $\R^d$; at this time, the hitting
particle instantaneously jumps to the position of an other particle;
then the particles evolve as independent Brownian motions, until one
particle hits the boundary  $\partial D$ and so on. They suggested that
the empirical distribution of the process was closely related to the
distribution of a Brownian motion conditioned not to hit the boundary
of $D$ and that the long-time behavior of its empirical distribution
was an approximation of the so-called \textit{quasi-stationary distribution} of
the Brownian motion absorbed at $\d D$.

\me
Two questions naturally arise in the study of interacting particle systems with jump from the boundary. The first one (and probably the most challenging one) concerns the non-explosion of the number of jumps in finite time. This question has been studied by Burdzy,
Ho{\l}yst and March~\cite{Burdzy2000}, by L\"obus~\cite{Lobus2009} and
by Bieniek, Burdzy and Finch~\cite{Bieniek2009}. Different dynamics
than the Brownian one have also been studied by Grigorescu and
Kang~\cite{Grigorescu2011}, where diffusive particles with
$C^{\infty}$ coefficients are considered, and by
Villemonais~\cite{Villemonais2010}, \cite{Villemonais2011} (where
time-inhomogeneous and environment-dependent diffusive particles are
considered). Let us emphasize that, in the last three articles, different jump mechanisms are
also allowed: each hitting particle can jump from the boundary to a position which follows a distribution
that depends on the position of the other particles (and on time and on the
environment for the last paper) in a general way, which is also the setting of the present paper. Concerning the non-explosion of the number of jumps, it is
remarkable that non-trivial counter-example (cases with explosion of the number of jumps in finite time) remained unkown up to the recent
results of Bieniek, Burdzy and Pal~\cite{Bieniek2011}.

\me
The second challenging question which naturally arises in this setting is: when the particle system is well defined (that is when the number of jumps doesn't explode), do the particles degenerate to the boundary in the long term? These question has already been treated in the case of the Fleming-Viot type particle system described above when the particles evolve as Brownian motions~\cite{Burdzy2000} or Brownian motions with drift~\cite{Villemonais2010}. However, the question remains open for more general jump mechanisms and different dynamics.
In the present paper, we answer this question by providing a criterion ensuring that the laws of the empirical distribution of a general particle system are uniformly tight.
This criterion is directly derived from the non-explosion criterion provided in~\cite{Villemonais2011}, which allows various choices of 
the dynamics and jump mechanism of the particles.

\me
As it has been suggested in~\cite{Burdzy1996}, the  particular case of the Fleming-Viot type particle systems (with jump from the boundary to the position of another particle) provides an approximation method for the conditional distribution of diffusion processes conditioned not to hit a boundary. This property, which has been proved 
by Grigorescu and Kang
in~\cite{Grigorescu2004} for Brownian particles, by
Villemonais~\cite{Villemonais2010} for Brownian particles with
drift, has been proved in all generality in~\cite{Villemonais2011}
for general Markov processes. Numerical implementations of this method have been used in~\cite{Villemonais2010} to provide numerical approximation of quasi-stationary distributions and in~\cite{Meleard2011} to compute the distance between the quasi-stationary distribution of a process and its conditional distribution. In the present paper, we use this convergence result to obtain theoretical informations on the distribution of diffusion processes conditioned not to reach the boundary of an open subset $D$ of $\mathbb{R}^d$. More precisely, we derive the uniform tightness of the distribution of such diffusion processes from the uniform tightness of the empirical distributions of a sequence of approximating particle systems.

\me The paper is divided into three parts. In Section~\ref{section:definition-particle-system}, we describe precisely the dynamic of the particle systems we are interested in. In particular, we recall a criterion taken from~\cite{Villemonais2011} which ensures that the particle systems are well defined (\textit{i.e.} that the number of jumps of the system remains bounded in finite time almost surely). In Section~\ref{section:uniform-tightness-empirical}, we state and prove our main result, providing a criterion which ensures that the laws of the empirical distributions of the particle systems are uniformly tight. In Section~\ref{section:uniform-tightness-distribution}, we apply this result to the Fleming-Viot type process introduced in~\cite{Burdzy1996}, which allows us to derive the uniform tightness of the distribution of diffusion processes conditioned not to reach the boundary of a bounded open subset of $\mathbb{R}^d$.

\section{Definition of the particle systems}
\label{section:definition-particle-system}
For each $N\geq 2$, we define an $N$-particles system whose particles
evolve as independent time-inhomogeneous environment-dependent
diffusion processes between their jumps. In a first time we define the diffusion processes which will drive the particles between the jumps; in a
second time, we define the jump measures that will give the jump
positions of the particles. Finally, we recall a recent result~\cite{Villemonais2011} which ensures that the particle system is well
define for any time $t\geq 0$.

\me 
Let $E$ and $D$ be two bounded open subsets of $\mathbb{R}^{d}$
and $\mathbb{R}^{d'}$ respectively. For each $N\geq 2$, let ${\cal Z}^{1,N},...,{\cal Z}^{N,N}$ be a
family of $N$ strong Markov processes, each of them being equal to a
$3$-tuple $(t,e^{i,N}_t,Z^{i,N}_t)_{t\in[0,\tau_{\partial}[}$ which
    evolves in $\mathbb{R}_+\times E \times D$ as a time-inhomogeneous environment-dependent diffusion process.  In the
    $3$-tuple $(t,e^{i,N}_t,Z^{i,N}_t)$, the parameter $t$ denotes the
    time, $e^{i,N}_t\in E$ denotes the state of the environment and
    $Z^{i,N}_t\in D$ denotes the actual position of the diffusion.  Each diffusion
 process ${\cal Z}^{i,N}$ will be used to define the dynamic of the
 $i^{th}$ particle of the system between its jumps.
By a \textit{time-inhomogeneous environment-dependent diffusion
  process}, we mean that, for any $N \geq 2$ and any
$i\in\{1,...,N\}$, there exist four measurable functions
 \begin{equation*}
   \begin{split}
     &s^N_i:[0,+\infty[\times E\times D\mapsto
     \mathbb{R}^{d}\times\mathbb{R}^{d}\\
     &m^N_i:[0,+\infty[\times E\times D\mapsto
     \mathbb{R}^{d}\\
     &\sigma^N_i:[0,+\infty[\times E\times D\mapsto
     \mathbb{R}^{d'}\times\mathbb{R}^{d'}\\
     &\eta^N_i:[0,+\infty[\times E\times D\mapsto \mathbb{R}^{d'},
  \end{split}
 \end{equation*}
 such that ${\cal Z}^{i,N}=(.,e^{i,N},Z^{i,N})$ is solution to the
 stochastic differential system
 \begin{equation*}
   \begin{split}
     de^{i,N}_t&=s^N_i(t,e^{i,N}_t,Z^{i,N}_t)d\beta^{i,N}_t+m_i(t,e^{i,N}_t,Z^{i,N}_t)dt,\ \quad e^{i,N}_0\in E,\\
     dZ^{i,N}_t&=\sigma^N_i(t,e^{i,N}_{t},Z^i_{t})dB^{i,N}_t+\eta^{i,N}({t},e^{i,N}_{t},Z^{i,N}_{t})
     dt,\ \quad Z^{i,N}_0\in D,
   \end{split}
 \end{equation*}
 where the $(\beta^{i,N},B^{i,N})$ are independent standard $d+d'$ Brownian
 motions.  Each process ${\cal
   Z}^{i,N}$ is killed when $Z^{i,N}_t$ hits $\partial D$ (\textit{hard killing}) and with a rate of killing
 $\kappa^N_i(t,e^{i,N}_t,Z^{i,N}_t)\geq 0$ (\textit{soft killing}), where
 \begin{equation*}
   \kappa^N_i:[0,+\infty[\times E\times D\mapsto \mathbb{R}_+
 \end{equation*}
 is a uniformly bounded measurable function.  We emphasize that each
 process ${\cal Z}^{i,N}$ evolves in the same state space
 $\mathbb{R}_+\times E \times D$, for any $i,N$.

 \bi Let us now define the jump measures, given by two measurable functions
 \begin{equation*}
   {\cal
   S}^{N}:[0,+\infty[\times E^N \times D^N\rightarrow {\cal
       M}_1(E^N\times D^N)
 \end{equation*}
 and
 \begin{equation*}
   {\cal H}^N:[0,+\infty[\times
         E^N \times \partial (D^N)\rightarrow {\cal
           M}_1(E^N\times D^N),
 \end{equation*}
 where ${\cal M}_1(F)$ denotes the space of probability measures on $F$, for any set $F$.
 
\medskip \noindent The position of the particle system at time $t$ will be denoted by $(t,\mathbb{O}^N_t,\mathbb{X}^N_t)$, where $t$ is the time, $\mathbb{O}^N_t=(o^{1,N}_t,\cdots,o^{N,N}_t)\in E^N$ is the vector of environments and $\mathbb{X}^N_t=(X^{1,N}_t,\cdots,X^{N,N}_t)\in D^N$ is the vector of positions of the particles. In particular, the $i^{th}$ particle at time $t\geq 0$ is given by $(t,o^{i,N}_t,X^{i,N}_t)$. We are now able to describe precisely the dynamic of our particle system.
 
 \bigskip \noindent \textbf{Dynamic of the particle system }
 The particles of the
 system evolve
 as independent copies of ${\cal Z}^{i,N}$, $i=1,...,N$, until one of
 them is killed. If the killed particle is softly killed, then the whole system
 jumps instantaneously with respect to the jump measure ${\cal
   S}^{N}(t,\mathbb{O}^N_t,\mathbb{X}^N_t)$; if the killed particle is hardly
 killed, it jumps with respect to ${\cal
   H}^N(t,\mathbb{O}^N_t,\mathbb{X}^N_t)$. Then the particles evolve
 as independent copies of ${\cal Z}^{i,N}$, $i=1,...,N$, until one of
 them is killed and so on.

 \bi We denote the successive jump times of the particle system by
 \begin{equation*}
   \tau^N_1 < \tau^N_2 < ... <\tau^N_n <...,
 \end{equation*}
 and we set $\tau^N_{\infty}=\lim_{n\rightarrow\infty} \tau^N_n$. It is clear that the $N$-particles system is well defined at
 any time $t$ such that $t<\tau^N_{\infty}$, but there is no natural way
 to extend the definition of the particle system beyond the explosion time
 $\tau^N_{\infty}$. As a consequence, it is of first important to be able to decide
 whether the number of jumps of the process explodes in finite time or
 not, that is to decide wether $\P(\tau^N_{\infty}=+\infty)<1$ or
 $=1$. The two following hypotheses ensure that $\P(\tau^N_{\infty}=+\infty)=1$ holds for any $N\geq 2$.
  As mentioned in the introduction, different criterions ensuring this non-explosion property exist, but they
 do not allow time-inhomogeneous (and
 \textit{a fortiori} environment-dependent) diffusive particles; moreover, they are far more restrictive in the choice of the jump measures. 
 
\medskip \noindent The first assumption concerns the processes ${\cal Z}^{i,N}$, $i=1,\ldots,N$, which give the dynamics of the particles between the jumps. We denote by $\phi_D$ the Euclidean distance to
 the boundary $\partial D$, defined for all $x\in \mathbb{R}^{d'}$ by
 \begin{equation*}
   \phi_D(x)=\inf_{y\in \partial D} \|x-y\|_2,
 \end{equation*}
 where $\|.\|_2$ denotes the Euclidean norm of $\mathbb{R}^{d'}$. For any $a>0$, the boundary's neighbourhood $D^a$ of $\partial D$ is defined by
 $$
 	D^a=\{x\in D,\ \phi_D(x)<a\}.
 $$

\begin{hypothesis}
  \label{hypothesis:particle-dynamics}
  We assume that there exist five positive constants $a_0$, $A$, $k_g$, $c_{0}$ and $C_{0}$ such that
  \begin{enumerate}
  \item $\phi_D$ is of class $C^2$ on $D^{a_0}$, with uniformly bounded derivatives,
  \item for any $N\geq 2$, $\kappa_{i}^N$ is uniformly bounded by $A$ on $[0,+\infty[\times E
      \times D$ and $s^{N}_i,\sigma^{N}_i,m^{N}_i$ and $\eta^{N}_i$
      are uniformly bounded by $A$ on $[0,+\infty[\times E \times D^{a_0}$,
      \item for any $N\geq 2$ and any  $i\in\{1,...,N\}$, there exist two
    measurable functions $f^{N}_i:[0,+\infty[\times E\times D^{a_0}\rightarrow
    \mathbb{R}_+$ and $g^{N}_i:[0,+\infty[\times E\times D^{a_0}\rightarrow
    \mathbb{R}$ such that $\forall (t,e,z)\in [0,+\infty[\times
    E\times D^{a_0}$,
    \begin{equation}
      \label{chapitre4:EqHyThExDriftTerm}
      \sum_{k,l}\frac{\partial \phi_D}{\partial x_k}(z) \frac{\partial
        \phi_D}{\partial x_l}(z)
      [\sigma^{N}_i\sigma^{N*}_i]_{kl}(t,e,z)=f^{N}_i(t,e,z)+g^{N}_i(t,e,z),
    \end{equation}
    and such that
    \begin{enumerate}
    \item $f^{N}_i$ is of class $C^1$ in time and of class $C^2$ in
      environment/space, and the derivatives of $f^{N}_i$ are
      uniformly bounded by $A$ in $[0,+\infty[\times E\times D^{a_0}$.
    \item for all
      $(t,e,z)\in[0,+\infty[\times E \times D^{a_0}$,
      \begin{equation*}
	|g^{N}_i(t,e,z)|\leq k_g\phi_D(z),
      \end{equation*}
    \item for all $(t,e,z)\in[0,+\infty[\times E \times D^{a_0}$,
      \begin{equation*}
        c_{0}<f^{N}_i(t,e,z)<C_{0}\text{ and }c_{0}<f^{N}_i(t,e,z)+g^{N}_i(t,e,z)<C_{0}.
      \end{equation*}
    \end{enumerate}
  \end{enumerate}
\end{hypothesis}

\medskip \noindent The second assumption below concerns the jump measures ${\cal H}^N$ and ${\cal S}^N$. Let us first remark that, when the system hits the boundary $\partial(D^N)$, at most one particle hits $\partial D$. This implies that the whole set of particles hits one and only one of the sets ${\cal D}^N_i$, $i=1,\cdots,N$, defined by
\begin{equation*}
{\cal D}_i=\left\lbrace (x_1,\cdots,x_N)\in\partial (D^N),\ x_i\in \partial D\ \mbox{and}\ x_j\in D,\forall j\neq i \right\rbrace.
\end{equation*}
With this definition, it is clear that the $i^{th}$ particle is hardly killed if and only if $(\cdot,\mathbb{O}^N,\mathbb{X}^N)$ hits $[0,+\infty[\times E^N\times {\cal D}^N_i$. In particular, the behavior of ${\cal H}^N$ outside these sets does not present any interest.

 \begin{hypothesis}
 \label{hypothesis:jump-measure}
 We assume that, for any $N\geq 2$,
     \begin{enumerate}
  \item There exists a non-decreasing continuous function
    $h^N:\mathbb{R}_+\rightarrow\mathbb{R}_+$ vanishing only at $0$ such that,
    $\forall i\in\{1,...,N\}$,
    \begin{equation*}
      \inf_{(t,e,(x_1,...,x_N))\in[0,+\infty[\times E^N\times {\cal D}^N_i}{\cal H}^N(t,e,x_1,...,x_N)(E^N\times A^N_{i})\geq p^N_0,
    \end{equation*}
    where $p^N_0>0$ is a positive constant and $A^N_{i}\subset D^N$ is the
    set defined by
    \begin{equation*}
      A^N_{i}=\left\lbrace (y_1,...,y_N)\in D^N\,|\,\exists j\neq i\text{ such that }\phi_{D}(y_{i})\geq h^N(\phi_D(y_j)) \right\rbrace.
    \end{equation*}
    
  \item We have
    \begin{equation*}
      \inf_{(t,e,(x_1,...,x_N))\in[0,+\infty[\times E\times {\cal D}_i}{\cal H}(t,e,x_1,...,x_N)(E\times B_{x_1,...,x_n})=1,
    \end{equation*}
    where
    \begin{equation*}
      B_{x_1,...,x_n}=\left\lbrace (y_1,...,y_N)\in D^N\,|\,\forall i,\  \phi_{i}(y_{i})\geq \phi_i(x_i) \right\rbrace
    \end{equation*}
  \end{enumerate}
 \end{hypothesis}
 
\medskip \noindent Let us explain the meaning of each point of the last hypothesis.
 
\medskip  \noindent - The set $A^N_i$ is a subset of $D^N$ such that if the $i^{th}$ component of an element $(y_1,\cdots,y_n)\in A^N_i$ is near the boundary $\partial D$ (that is if $\phi_D(y_i)\ll 1$), then at least one another component $y_j$ fulfils $h^N(\phi_D(y_j))\ll 1$ and thus $\phi_D(y_j)\ll 1$. This implies that if the $i^{th}$ particle jumps after a hard killing to a position located near the boundary $\partial D$, then, with a probability lowered by $p_0^N$, at least one of the other particles is located near the boundary.

\medskip \noindent - The set $B_{x_1,...,x_n}$ is a subset of $D^N$ such that for any $(y_1,\cdots,y_n)\in B_{x_1,...,x_n}$, the components $y_i$ are respectively farther from the boundary than $x_i$. This means that after a hard killing, each particle jumps to a position farther from the boundary after the jump than before the jump.

\bigskip \noindent Let us now recall the result which will ensure that the interacting particle
systems are well defined, for any $N\geq 2$.

 \begin{theorem}[V. 2011 in~\cite{Villemonais2011}]
 \label{theorem:non-explosion}
   Assume that Hypotheses~\ref{hypothesis:particle-dynamics} and~\ref{hypothesis:jump-measure} are fulfilled, then, for any $N\geq 2$, the number
   of jumps of the $N$-particles system doesn't explode in finite time almost
   surely. Equivalently, we have $\tau^N_{\infty}=+\infty$ almost surely, for all $N\geq 2$.
 \end{theorem}
 
\medskip \noindent Before turning to the next step of our study,
 let us emphasize that one could consider a more complicated situation
 where the diffusion process ${\cal Z}^{i,N}$ is reflected on $\partial D$ and killed when its local time on the boundary reaches the value
 of an independent exponential random variable. Such processes have
 been studied in deep details in the one-dimensional situation
 $d'_0=1$ (see for instance~\cite{Kolb2010} and references therein) and in the
 multi-dimensional situation (see for instance~\cite{Zhen2010}). The technics
 and calculous used in~\cite{Villemonais2011} would still be valid, as would be the rest of
 the present paper. However, the case under study is complicated
 enough, so that we only consider killing boundaries without reflection.

\section{Uniform tightness for time inhomogeneous particle systems}
\label{section:uniform-tightness-empirical}
In the present section, we consider a family of interacting
particle systems and prove a sufficient criterion for the
uniform tightness of the laws of the empirical distributions of the
particle system, at any time time $t$.

\bi For any $N\geq 2$, let $(\cdot,\mathbb{O}^{(N)},\mathbb{X}^{(N)})$
be a $N$-particles system defined as in the previous section, driven
by the time-inhomogeneous environment-dependent diffusion processes
${\cal Z}^{i,N}$, $i=1,...,N$, with environment state space $E$ and
diffusion state space $D$. We also denote by ${\cal S}^N$ and ${\cal
  H}^N$ the jump measures of the interacting particle system.

\bi For any time $t<\tau^N_{\infty}$, we define the empirical
distribution of the particle system with $N$ particles 
$(\cdot,\mathbb{O}^{(N)},\mathbb{X}^{(N)})$ at time $t$ by
\begin{equation*}
  \mu^N_t=\frac{1}{N}\sum_{i=1}^N \delta_{X^{i,N}_t},
\end{equation*}
which is a probability measure on $D$ (we emphasize that the empirical
distribution doesn't take into account the value of the environment).
We're now able to state our tightness result (which will be used in the next section to prove that the family of conditional
distributions of time-homogeneous environment-dependent diffusion
processes is uniformly tight).

\begin{theorem}
\label{theorem:uniform-tightness}
Assume that Hypotheses \ref{hypothesis:particle-dynamics} and \ref{hypothesis:jump-measure} are
fulfilled and fix $t_0>0$.  Then, for any $\epsilon>0$, there exists $a_{\epsilon}>0$ and $N_{\epsilon}\geq 2$ such that, independently of the initial distribution of the particle systems,
$$
\E\left(\mu^N_T(D^{a_{\epsilon}})\right)\leq \epsilon,\ \forall N\geq N_{\epsilon},\ \forall T\geq t_0.
$$
In particular, for any sequence of deterministic times $(t_N)_{N\geq 2}$ such that $\inf_{N\geq 2}{t_N}>0$, the family of laws of the
random measures $\mu^{N}_{t_N}(dx)$, $N\geq 2$, is uniformly
tight.
\end{theorem}

\medskip \noindent We emphasize that it is not required for the laws of the initial empirical distributions $(\mu^N_0)_{N\geq 2}$ to be uniformly tight.

\medskip \noindent We also recall that the uniform tightness of the laws of the empirical distributions is of particular interest since it implies that the given sequence is weakly sequentially compact, as a sequence of random probability measures.

\begin{proof}[Proof of Theorem~\ref{theorem:uniform-tightness}]
  By Theorem~\ref{theorem:non-explosion}, the 
  particle system
  $(\cdot,\mathbb{O}^{(N)}_{\cdot},\mathbb{X}^{(N)}_{\cdot})$ is well
  defined at any time for any $N\geq 2$.
  
  \bigskip \noindent The proof of the uniform tightness is based on the following criterion, due to Jakubowski in~\cite{Jakubowski1988}: if, for any $\epsilon>0$, there exists a compact subset $K_{\epsilon}$ of $D$ such that $\mathbb{E}(\mu^N_{t_N}(K_{\epsilon}))\geq 1-\epsilon$ for all $N\geq 2$, then the family of laws of the probability measures $\mu^N_{t_N}$ is uniformly tight. 
Taking $t_0=\inf_{N\geq 2} t_N>0$, the first part of Theorem~\ref{theorem:uniform-tightness} and Jakubowski's criterion clearly implies the second part of Theorem~\ref{theorem:uniform-tightness}. Thus it remains us to prove the first part of the result.

  \bigskip \noindent Fix $\epsilon>0$ and $t_0>0$. In a first
  time, we assume that the killing rate $\kappa_i^{N}$ is equal to $0$, for all
  $N\geq 2$ and $i\in\{1,...,N\}$.  For any $N\geq 2$ and any
  $a>0$, we have
  \begin{equation*}
    \begin{split}
      E\left( \mu^{N}_{t_0}(D^{a}) \right)^2& \leq E\left( \mu^{N}_{t_0}(D^{a})^2
      \right),\\
       & \leq E\left(
      \frac{1}{N^2}\sum_{i=1}^N\left(\delta_{X^{i,N}_{t_0}}(D^{a})\right)^2+\frac{1}{N^2}\sum_{1\leq
      i\neq j \leq N}\delta_{X^{i,N}_{t_0}}(D^{a})\delta_{X^{j,N}_{t_0}}(D^{a})
      \right)\\
      & \leq \frac{1}{N}E\left( \mu^{N}_{t_0}(D^{a})\right) + \frac{1}{N^2}\sum_{1\leq
      i\neq j \leq N}E\left(
      \delta_{X^{i,N}_{t_0}}(D^{a})\delta_{X^{j,N}_{t_0}}(D^{a})\right)\\
      & \leq \frac{1}{N}E\left( \mu^{N}_{t_0}(D^{a})\right)+\max_{1\leq i\neq j \leq N}
      P\left(\phi_D(X^{i,N}_{t_0})\leq a \text{ and }
      \phi_D(X^{j,N}_{t_0}) \leq a\right).
    \end{split}
  \end{equation*}

  \noindent Fix $N\geq 2$ and $i\neq j\in \{1,...,N\}$. For all
  $\gamma\in[0,\frac{a_0}{2}]$ (where $a_0$ is taken from Hypothesis~\ref{hypothesis:particle-dynamics}), we define the stopping time
  \begin{equation*}
    S^{i,N}_{\gamma}=\inf\{t\geq 0,\ \phi_D(X^{i,N}_t)\geq \gamma\},
  \end{equation*}
  which is the first time at which the distance between the $i^{th}$ particle and the boundary $\partial D$ is greater than $\gamma$.
  Let us now state the following useful Lemma, whose proof  
  is
  postponed to the end of this section.
  \begin{lemme}
    \label{chapitre4:lemma:1-tightness-inhomogeneous-case}
    There exists a positive constant $C>0$, independent of $i,j,N$ and
    $\gamma$, such that for all $a\in[0,\gamma\sqrt{\frac{c_0}{C_0}}[$ (where $c_0$ and $C_0$ are taken from Hypothesis~\ref{hypothesis:particle-dynamics}),
    \begin{equation*}
      P\left(\exists t\in [S^{i,N}_{\gamma},t_0],\,
      \sqrt{\phi_{D}(X^{i,N}_{t})^2+\phi_{D}(X^{j,N}_{t})^2}\leq a
      \right)\\
      \leq
      \frac{Ct_0}{\log \left(\frac{\gamma}{a}\sqrt{\frac{c_0}{C_0}}\right)}.
  \end{equation*}
  \end{lemme}
  \noindent We immediately deduce from
  Lemma~\ref{chapitre4:lemma:1-tightness-inhomogeneous-case} that, for any $\gamma<a_0/2$ and any
  $a\in[0,\gamma \sqrt{\frac{c_0}{C_0}}[$,
  \begin{equation*}
    P\left(\phi_D(X^{i,N}_{t_0})\leq \frac{a}{\sqrt{2}} \text{ and }
    \phi_D(X^{j,N}_{t_0}) \leq \frac{a}{\sqrt{2}}\right)\leq
    \frac{C{t_0}}{\log \left(\frac{\gamma}{a}\sqrt{\frac{c_0}{C_0}}\right)}+P(S^{i,N}_{\gamma}> {t_0}).
  \end{equation*}
  In particular, replacing $a$ by $a\sqrt{2}$, we deduce that, for any $\gamma<a_0/2$ and any
  $a\in[0,\gamma \sqrt{\frac{c_0}{2C_0}}[$,
  \begin{equation}
    \label{chapitre4:equation:bound-expectation-gamma-dependent}
    E\left( \mu^N_{t_0}(D^{a}) \right)^2 \leq \frac{1}{N}E\left(
    \mu^N_{t_0}(D^{a})\right)+\frac{C{t_0}}{\log
      \left(\frac{\gamma}{a}\sqrt{\frac{c_0}{2C_0}}\right)}+\max_{1\leq i \leq N}
    P\left(S^{i,N}_{\gamma}> {t_0}\right).
  \end{equation}
  The proof of the following lemma is also postponed to the end of this section.
  \begin{lemme}
    \label{chapitre4:lemma:bound-on-S-gamma}
    For any $\epsilon>0$, there exists a constant
    $\gamma_{\epsilon}>0$ such that, for all $N\geq 2$ and all
    $i\in\{1,...,N\}$, we have
    \begin{equation}
      \label{chapitre4:equation:bound-on-proba-S-gamma-to-prove}
      P\left(S^{i,N}_{\gamma_{\epsilon}}> {t_0}\right)\leq
      \epsilon/3,
    \end{equation}
    independently of the sequence of initial distributions.
  \end{lemme}

  \noindent By
  \eqref{chapitre4:equation:bound-expectation-gamma-dependent} and Lemma~\ref{chapitre4:lemma:bound-on-S-gamma}, we
  deduce that, for all $a\in[0,\gamma_{\epsilon}\sqrt{\frac{c_0}{2C_0}}[$,
  \begin{equation*}
      E\left( \mu^N_{t_0}(D^{a}) \right)^2 \leq \frac{1}{N}E\left(
    \mu^N_{t_0}(D^{a})\right)+\frac{C{t_0}}{\log
      \left(\frac{\gamma_{\epsilon}}{a}\sqrt{\frac{c_0}{2C_0}}\right)}+\frac{\epsilon}{3}.
  \end{equation*}
  \noindent Fixing an integer $N_{\epsilon}\geq 2 \vee \frac{3}{\epsilon}$,
  we thus have, for all $N\geq N_{\epsilon}$ and all $a>0$,
   \begin{equation*}
     \frac{1}{N}E\left( \mu^N_{t_0}(D^{a})\right)\leq \frac{\epsilon}{3}.
   \end{equation*}
   Let $a_{\epsilon}>0$ be a positive constant such that $\log
     \left(\frac{\gamma_{\epsilon}}{a_{\epsilon}}\sqrt{\frac{c_0}{2C_0}}\right)\leq
   3(C{t_0}\epsilon)^{-1}$. We then have
   \begin{equation*}
     E\left( \mu^N_{t_0}(D^{a_{\epsilon}}) \right)^2 \leq \epsilon,\ \forall N\geq N_{\epsilon},
   \end{equation*}
   independently of the sequence of initial distributions. We deduce that the first part of Theorem~\ref{theorem:uniform-tightness} is fulfilled for $T=t_0$.
   
   \bi Fix $T>{t_0}$. Since the previous inequality doesn't depend
   on the distribution of the initial position
   $(X^{1,N}_0,...,X^{N,N}_0)$, it can be applied to the process
   initially distributed with the same distribution as
   $(X^{1,N}_{T-{t_0}},...,X^{2,N}_{T-{t_0}})$. By the Markov property of the
    particle system, we thus obtain
   \begin{equation*}
     E\left( \mu^N_T(D^{a_{\epsilon}}) \right)^2 \leq \epsilon,\ \forall N\geq N_{\epsilon}.
   \end{equation*}
   This allows us to conclude the proof of the first part of
   Theorem~\ref{theorem:uniform-tightness} when $\kappa_i^{N}=0$
   for all $N\geq N_{\epsilon}$ and $i\in\{1,...,N\}$. 

\bi
  Fix $N\geq 2$ and assume now that $(\kappa_i^{N})_{i\in\{1,...,N\}}$
  isn't equal to $0$. Fix $i,j\in\{1,...,N\}$. For any $\gamma>0$, we
  define the stopping time $S^{i,N}_{\gamma}$ as above. We also denote
  by $\tau^{soft}_{\gamma}$ the first soft killing time of
  $X^{i,N}$ or $X^{j,N}$ after $S^{i,N}_{\gamma}$, that is
  \begin{equation*}
    \tau^{soft}_{\gamma}=\inf\{t\geq
    S^{i,N}_{\gamma},\ X^{i,N}\text{ or }X^{j,N}\text{ is softly
      killed at time }t\}.
  \end{equation*}
\noindent Similarly to Lemma
   \ref{chapitre4:lemma:1-tightness-inhomogeneous-case},
     we have, for any $\gamma<a_0/2$ and any $a\in [0,\gamma \sqrt{\frac{c_0}{C_0}}]$,
  \begin{equation*}
      P\left(\exists t\in [S_{\gamma},\tau^{soft}_{\gamma}\wedge {t_0}[,\,
      \sqrt{\phi_{D}(X^{i,N}_{t})^2+\phi_{D}(X^{j,N}_{t})^2}\leq a
      \right)\\
      \leq
      \frac{C{t_0}}{\log \left(\frac{\gamma}{a}\sqrt{\frac{c_0}{C_0}}\right)}.
  \end{equation*}
  In particular, we deduce that 
  \begin{multline*}
    P\left(\phi_D(X^{i,N}_{t_0})\leq a \text{ and }
    \phi_D(X^{j,N}_{t_0}) \leq a\right)\\
\leq \frac{C}{\log
      \left(\frac{\gamma}{a}\sqrt{\frac{c_0}{2C_0}}\right)}
    + P(S_{\gamma}>{t_0} \text{ and } \tau^{soft}_{\gamma}\geq {t_0}) + P(\tau^{soft}_{\gamma}<{t_0}).
  \end{multline*}
  By Hypothesis~\ref{hypothesis:particle-dynamics}, the killing rates $\kappa_i^{N}$ and
  $\kappa_j^{N}$ are uniformly bounded by a constant
  $A>0$. As a consequence, there exists $T_0>0$ such
  that, for all $t\leq T_0$ and all $\gamma>0$,
  \begin{equation*}
    P(\tau^{soft}_{\gamma}<t)\leq \frac{\epsilon}{4}.
  \end{equation*}
  We emphasize that $T_0$ is chosen so that it only depends on the
  uniform bound $A$ and that we can assume, without loss of generality, that $t_0$ is smaller than $T_0$.  By the
  same arguments as in the proof (postponed below) of
  Lemma~\ref{chapitre4:lemma:bound-on-S-gamma}, we can find $\gamma_{\epsilon}>0$
  such that
  \begin{equation*}
    P(S^{i,N}_{\gamma_{\epsilon}}>t_0 \text{ and } \tau^{soft}_{\gamma_{\epsilon}}\geq t_0)\leq \frac{\epsilon}{4}.
  \end{equation*}
  Finally, choosing $a_{\epsilon}$ small enough, we deduce that
  \begin{equation*}
    P\left(\phi_D(X^{i,N}_{t_0})\leq a_{\epsilon} \text{ and }
    \phi_D(X^{j,N}_{t_0} \leq a_{\epsilon}\right)\leq
    \frac{3\epsilon}{4}.
  \end{equation*}
 Finally, choosing $N_{\epsilon}\geq 2 \vee \frac{4}{\epsilon}$ and proceeding as in the first part of the proof, we deduce that  Theorem~\ref{theorem:uniform-tightness} holds for $T=t_0$ and can be extended by the Markov property to any $T>t_0$.
\end{proof}

\begin{proof}[Proof of Lemma~\ref{chapitre4:lemma:1-tightness-inhomogeneous-case}]
  Fix $\gamma\in[0,a_0/2]$ and let us prove that, for all $a\in[0,\gamma[$,
    \begin{equation*}
      P\left(\exists t\in [S^{i,N}_{\gamma},{t_0}],\,
      \phi_{D}(X^{i,N}_{t})+\phi_{D}(X^{j,N}_{t})\leq a
      \right)\\
      \leq
      \frac{C{t_0}}{\log \left(\frac{\gamma}{a}\right)}.
  \end{equation*}

 \bigskip \noindent Let $(s_n)_{n\geq 0}$ be the sequence of stopping
 times defined by
  \begin{equation*}
    s_0=\inf\{s\in [S^{i,N}_{\gamma},{t_0}],
    \ \sqrt{\phi_D(X^{i,N}_s)^2+\phi_D(X^{j,N}_s)^2}\leq a_0/2\}\wedge {t_0}
  \end{equation*}
  and, for all $n\geq 0$,
  \begin{equation*}
    \begin{split}
      s_{2n+1}&=\inf\{s\in [s_{2n},{t_0}],\ \sqrt{\phi_D(X^{i,N}_s)^2+\phi_D(X^{j,N}_s)^2}\geq a_0\}\wedge {t_0}\\
      s_{2n+2}&=\inf\{s\in [s_{2n+1},{t_0}],
      \ \sqrt{\phi_D(X^{i,N}_s)^2+\phi_D(X^{j,N}_s)^2}\leq a_0/2\}\wedge {t_0}.
    \end{split}
  \end{equation*}
  It is immediate that $s_n$ converges almost surely to ${t_0}$. By
  construction, we have for all $n\geq 0$,
  \begin{equation*}
    \left\lbrace
    \begin{array}{l}
      \phi_D(X^{i,N}_t)< a_0\text{ and }\phi_D(X^{j,N}_t)<a_0,\ \forall t\in[s_{2n},s_{2n+1}[,\\
      \sqrt{\phi_D(X^{i,N}_t)^2+\phi_D(X^{j,N}_t)^2}\geq a_0/2 \text{ otherwise.}
    \end{array}
    \right.
  \end{equation*}
  In particular, , for all $t\in[s_{2n},s_{2n+1}[$, $\phi_D$ is of class
  $C^2$ at $X^{i,N}_t$ and $X^{j,N}_t$ almost
      surely, by the first point of
  Hypothesis~\ref{hypothesis:particle-dynamics}. This will allow us to compute the It\^o's decomposition
      of $\phi_D(X^{i,N})$ and $\phi_D(X^{j,N})$ 
      at any time $t\in[s_{2n},s_{2n+1}[$, using It\^o's formula.
  
  \medskip\noindent
  For all $n\geq 0$, we have
  \begin{equation}
    \label{chapitre4:equation:lemma-useful-1}
    P\left(\exists t\in [s_{2n+1},s_{2n+2}[,\ \sqrt{\phi_D(X^{i,N}_t)^2 + \phi_D(X^{j,N}_t)^2} \leq a\right)= 0,\  \forall a < a_0/2.
  \end{equation}
  
  \medskip\noindent
  Fix $n\geq 0$ and let us now prove that there exists a constant $C>0$ such that
  \begin{equation}
  \label{equation:yi-justification}
    P\left(\exists t\in [s_{2n},s_{2n+1}[,\ \phi_D(X^{i,N}_t)\leq
        a \text{ and } \phi_D(X^{j,N}_t) \leq a\right)
        \leq
        \frac{C}{\log \left(\frac{\gamma}{a}\right)}
        E\left(s_{2n+1}-s_{2n}\right).
  \end{equation}
  We define the positive semi-martingale $Y^i$ by
  \begin{equation}
    \label{chapitre4:EqThExYdecomp}
    Y^i_t=\left\lbrace
    \begin{array}{l}
      \phi_D(X^{i,N}_{s_{2n}+t}) \text{ if } t<s_{2n+1}-s_{2n},\\
      a_0/2+|W^i_t| \text{ if } t\geq s_{2n+1}-s_{2n},
    \end{array}
    \right.
  \end{equation}
  where $W^i$ is a standard one dimensional Brownian motion independent of the rest of the process.
  The extension after time $s_{2n+1}-s_{2n}$ (we recall that $n$ is fixed here)
  allows us to define $Y^i_t$ at any time $t\in[0,+\infty[$. We
  define similarly the semi-martingale $Y^j$. The inequality~\eqref{equation:yi-justification} is proved using \cite[Proposition~4.1]{Villemonais2011} applied to the pair of semi-martingale $Y^i, Y^j$.
  In order to do so, we need the It\^o's
  decompositions of $Y^i$ and $Y^j$. Let us set
  \begin{equation*}
    \pi^i_t=
    \begin{cases}
      f^{N}_i(s_{2n}+t,o^{i,N}_{s_{2n}+t},X^{i,N}_{s_{2n}+t}),&\text{if } 0\leq t<s_{2n+1}-s_{2n},\\
      1,& \text{if } t\geq s_{2n+1}-s_{2n}
    \end{cases}
  \end{equation*}
  and
  \begin{equation*}
    \rho^i_t=
    \begin{cases}
      g^{N}_i({s_{2n}+t},o^{i,N}_{s_{2n}+t},X^{i,N}_{s_{2n}+t}),&\text{if } t<s_{2n+1}-s_{2n},\\
      0,&\text{if } t\geq s_{2n+1}-s_{2n},
    \end{cases}
  \end{equation*}
  where $f^{N}_i$ and $g^{N}_i$ are given by
  Hypothesis~\ref{hypothesis:particle-dynamics}.  By the It\^o's
  formula applied to $Y^i$, we have
  \begin{equation*}
    dY^i_t=dM^i_t+b^i_t dt + dK^i_t + Y^i_t-Y^i_{t\minus},
  \end{equation*}
  where $M^i$ is a local martingale such that
  \begin{equation*}
    d\langle M^i \rangle_t=(\pi^i_t+\rho^i_t)dt;
  \end{equation*}
    $b^i$ is the adapted process given, if $t<s_{2n+1}-s_{2n}$, by
  \begin{multline*}
    b^i_t=
      \sum_{k=1}^{d'i}\frac{\partial \phi_i}{\partial
        x_k}(X^{i,N}_{t_{2n}+t})[\eta^N_i]_k({t_{2n}+t},o^{i,N}_{t_{2n}+t},X^{i,N}_{t_{2n}+t})\\
      +\frac{1}{2}\sum_{k,l=1}^{d_i}\frac{\partial^2 \phi}{\partial
        x_k\partial
        x_l}(X^{i,N}_{t_{2n}+t})[\sigma^N_i\sigma^{N*}_i]_{kl}({t_{2n}+t},o^{i,N}_{t_{2n}+t},X^{i,N}_{t_{2n}+t}),
  \end{multline*}
      and, if $t\geq s_{2n+1}-s_{2n}$, by $b^i_t=0$ ; $K^i$ is a
      non-decreasing process given by the local time of $|W_t|$ at $0$
      after time $s_{2n+1}-s_{2n}$.  By the $3^{th}$ point of
      Hypothesis~\ref{hypothesis:particle-dynamics}, we
      have, for all $t\geq 0$,
  \begin{equation}
    \label{chapitre4:EqThExBornePiRho}
    c_{0}\wedge 1\leq \pi^i_t+\rho^i_t \leq C_{0}\vee 1,\text{ and
    } |\rho^i_t|\leq k_g Y^i_t
  \end{equation}
  By Hypothesis~\ref{hypothesis:particle-dynamics}, $\phi_D$ is of
  class $C^2$ on $D^{a_0}$, with uniformly bounded derivatives, and $\eta^N_i,\sigma^N_i$ are uniformly
  bounded. This implies that there exists $b_{\infty}>0$ (independent of $i$ and $N$) such that,
  for all $t\geq 0$,
  \begin{equation}
     \label{chapitre4:EqThExBorneB}
     b^i_t\geq -b_{\infty}.
  \end{equation}
  Similarly, we get the decomposition of $Y^j$, with $\pi^j$, $\rho^j$
  and $b^j$ fulfilling inequalities \eqref{chapitre4:EqThExBornePiRho} and
  \eqref{chapitre4:EqThExBorneB} (without loss of generality, we keep the same
  constants $c_{0}$, $C_{0}$, $k_g$ and $b_{\infty}$).

\me  Let us now compute the It\^o's decompositions of $\pi^i$ and $\pi^j$.
    We deduce from the It\^o's formula that there exist a local
  martingale $N^i$ and a finite variational process $L^i$ such that, for all $t\geq 0$,
   \begin{equation*}
    d\pi^i_t=dN^i_t+ dL^i_t + \pi^i_t-\pi^i_{t\minus},
  \end{equation*}
   where, for all $t\in [0,s_{2n+1}-s_{2n}[$,
   \begin{equation*}
       L^i_{t}=\int_0^t \left(\sum_{k=1}^{d_0}\frac{\partial f_i^{N}}{\partial
         e^k}(s_{2n}+s,o^i_{s_{2n}+s},X^i_{s_{2n}+s})+\sum_{k=1}^{d'_0}\frac{\partial
         f_i^{N}}{\partial x^k}(s_{2n}+s,o^i_{s_{2n}+s},X^i_{s_{2n}+s})\right) ds.
   \end{equation*}
   By Hypothesis~\ref{hypothesis:particle-dynamics}, the derivatives
   of $f_i^{N}$ are uniformly bounded, so that there exists a
   constant $C_L$ such that
   \begin{equation}
     \label{chapitre4:equation:bound-on-L}
     E\left(|L^i|_{s_{2n+1}-s_{2n}}\right)\leq C_L E\left(s_{2n+1}-s_{2n}\right).
   \end{equation}
   Let us set, for all $t<s_{2n+1}-s_{2n}$,
   \begin{multline*}
     \xi^i_t=\sum_{k=1,l}^{d_i}\frac{\partial f^N_i}{\partial e_k}(t,o^{i,N}_t,X^{i,N}_t)\frac{\partial f^N_i}{\partial e_l}(t,o^{i,N}_t,X^{i,N}_t)
                                            [s^N_i s_i^{N*}]_{kl}(t,o^{i,N}_t,X^{i,N}_t)\\
           +\sum_{k=1,l}^{d'_i}\frac{\partial f^N_i}{\partial x_k}(t,o^{i,N}_t,X^{i,N}_t)\frac{\partial f^N_i}{\partial x_l}(t,o^{i,N}_t,X^{i,N}_t)
                                            [\sigma^N_i \sigma_i^{N*}]_{kl}(t,o^{i,N}_t,X^{i,N}_t)
   \end{multline*}
   and, for all $t\geq s_{2n+1}-s_{2n}$, $\xi^i_t=0$. Then we have
      \begin{equation*}
     \langle N^i\rangle_t= \xi^i_t dt.
   \end{equation*}
   Thanks to the regularity assumptions on $f_i^N$ and the boundedness
   of $s^N_i$ and $\sigma^N_i$, there exists $C_{\xi}>0$ such that
  \begin{equation}
    \label{chapitre4:EqThExBorneXi}
    \xi^i_t \leq C_{\xi}.
  \end{equation}
  The same decomposition and inequalities hold for $\pi^j$, with the
  same constants $C_L$ and $C_{\xi}$. We emphasize that these
  constants are chosen independently of $i$, $j$ and $N$, since the
  bounds that we used are by assumption uniform in $i,j,N$.

  \bigskip\noindent We define the process
  \begin{equation*}
    \Phi_t\stackrel{def}{=}-\frac{1}{2}\log\left(\frac{(Y^i_t)^2}{\pi^1_t}+\frac{(Y^j_t)^2}{\pi^2_t}\right),\ t\geq 0,
  \end{equation*}
  and we set, for all $\epsilon>0$, $T_{\epsilon}=\inf\{t\in[0,T],\ \Phi_t\geq
  {\epsilon}^{-1}\}$. By the previous It\^o's decompositions, one can
  apply \cite[Proposition~4.1]{Villemonais2011} to the pair of semi-martingales $Y^1,Y^2$.
Thus, for any stopping time $\theta$, we have
  \begin{equation*}
    P\left(T_{\epsilon}\leq \theta \right)\leq
    \frac{1}{\epsilon^{-1}-\Phi_0}
    C\left(E(|L^i|_{\theta}+|L^j|_{\theta}) + E(\theta)\right).
  \end{equation*}
  Applying this result to $\theta=s_{2n+1}-s_{2n}$ (which is a
  stopping time for the filtration of the process $(X^{i,N},X^{j,N})$ after
  time $s_{2n}$) and using \eqref{chapitre4:equation:bound-on-L}, we deduce that
  there exists a constant $C'>0$, which only depend on the constants
  $b_{\infty},k_g,c_{0},C_{0},C_{\xi}$, such that
  \begin{equation}
    \label{chapitre4:equation:born-between-t-2n-and-t-2n-1}
    P\left(T_{\epsilon}\in [0,s_{2n+1}-s_{2n}[ \right)\leq
    \frac{1}{\epsilon^{-1}-\Phi_0} C'(2\,C_L+1)E\left(s_{2n+1}-s_{2n}\right).
  \end{equation}
 By Hypothesis~\ref{hypothesis:particle-dynamics}, we have
  \begin{equation*}
    \Phi_0\leq - \log \left(
    \sqrt{\frac{(Y^1_{0})^2}{C_{0}}+
    \frac{(Y^2_{0})^2}{C_{0}}} \right).
  \end{equation*}

\me  If $s_{2n}=T$, then $s_{2n+1}-s_{2n}=0$, so that $Y^1_{0}=Y^2_{0}=a_0$. If $s_{2n}<T$, then we have $s_{2n}=S^{i,N}_{\gamma}$ or $s_{2n}>S^{i,N}_{\gamma}$. If $s_{2n}=S^{i,N}_{\gamma}$, then, by definition of $S^{i,N}_{\gamma}$ and by the right continuity of the process, we have $\phi_D(X^{i,N}_{s_{2n}})\geq \gamma$, that is $Y^i_0\geq \gamma$. If $s_{2n}>S_{\gamma}^{i,N}$, then, by definition of $s_{2n}$,
  we have $$\sqrt{\phi_D(X^{i,N}_{s_{2n}\minus})^2+\phi_D(X^{j,N}_{s_{2n}\minus})^2}\geq a_0/2.$$
Since $\sqrt{\phi_D(X^{i,N})^2+\phi_D(X^{j,N})^2}$ can only have positive jumps, we deduce that
  $$
  \sqrt{\phi_D(X^{i,N}_{s_{2n}})^2+\phi_D(X^{j,N}_{s_{2n}})^2}=a_0/2\geq \gamma,
  $$
that is
  $\sqrt{(Y^1_0)^2+(Y^2_0)^2}\geq a_0/2\geq \gamma$.
  Finally, in all cases, we have
  \begin{equation*}
  \Phi_0\leq -\log \left(\frac{\gamma}{\sqrt{C_{0}}}\right).
  \end{equation*}
  Thus we deduce from~\eqref{chapitre4:equation:born-between-t-2n-and-t-2n-1} that, for $\epsilon>0$ small enough,
  \begin{equation*}
    P\left(T_{\epsilon}\in [0,s_{2n+1}-s_{2n}] \right) \leq
    \frac{1}{\epsilon^{-1}+\log \left(\frac{\gamma}{\sqrt{C_{0}}}\right)}
    C(2\,C_L+1)E\left(s_{2n+1}-s_{2n}\right).
  \end{equation*}
  which implies that
  \begin{multline*}
    P\left(\exists t\in [s_{2n},s_{2n+1}],\,
        \sqrt{\phi_{D}(X^{i,N}_{t})^2+\phi_{D}(X^{j,N}_{t})^2}\leq \sqrt{c_{0}} e^{-\frac{1}{\epsilon}}
          \right)\\
          \leq
          \frac{1}{\epsilon^{-1}+\log \left(\frac{\gamma}{\sqrt{C_{0}}}\right)}
          C(2\,C_L+1)E\left(s_{2n+1}-s_{2n}\right).
  \end{multline*}
  Replacing $\epsilon^{-1}$ by
  $-\log(a/\sqrt{c_{0}})$, we deduce that
  \begin{multline*}
    P\left(\exists t\in [s_{2n},s_{2n+1}],\,
        \sqrt{\phi_{D}(X^{i,N}_{t})^2+\phi_{D}(X^{j,N}_{t})^2}\leq a
          \right)\\
          \leq
          \frac{C(2\,C_L+1)}{\log \left(\frac{\gamma}{a}\sqrt{\frac{c_0}{C_0}}\right)}
          E\left(s_{2n+1}-s_{2n}\right).
  \end{multline*}
  Summing over $n\geq 0$ and using
  equality~\eqref{chapitre4:equation:lemma-useful-1}, we deduce that
  \begin{equation*}
    P\left(\exists t\in [S_{\gamma},t_0],\,
        \sqrt{\phi_{D}(X^{i,N}_{t})^2+\phi_{D}(X^{j,N}_{t})^2}\leq a
          \right)\\
          \leq
          \frac{C(2\,C_L+1)t_0}{\log \left(\frac{\gamma}{a}\sqrt{\frac{c_0}{C_0}}\right)}.
  \end{equation*}
 This immediately leads to Lemma~\ref{chapitre4:lemma:1-tightness-inhomogeneous-case}.
\end{proof}

\begin{proof}[Proof of Lemma~\ref{chapitre4:lemma:bound-on-S-gamma}] 

    In order to prove Lemma~\ref{chapitre4:lemma:bound-on-S-gamma}, we build a coupling
  between $\phi_D(X^{i,N})$ and a time changed reflected Brownian motion with drift.

 \me Let $(\theta_n)_{n\geq 0}$ be the sequence of stopping
  times defined by
  \begin{equation*}
    \theta_0=\inf\{t\in [0,T],\ \phi_D(X^{i,N}_t)\leq a_0/2\}\wedge T
  \end{equation*}
  and, for all $n\geq 0$,
  \begin{equation*}
    \begin{split}
      \theta_{2n+1}&=\inf\{t\in [\theta_{2n},T],\ \phi_D(X^{i,N}_t)\geq a_0\}\wedge T\\
      \theta_{2n+2}&=\inf\{t\in [\theta_{2n+1},T],\ \phi_D(X^{i,N}_t)\leq a_0/2\}\wedge T.
    \end{split}
  \end{equation*}
  It is immediate that $(\theta_n)$ converges almost surely to $T$ and
  that
  \begin{equation*}
    \begin{split}
      \phi_D(X^{i,N}_t)&\geq \frac{a_0}{2},\ \forall
      t\in[0,\theta_0]\ \text{and}\ 
      \forall t\in\cup_{n=0}^{\infty}[\theta_{2n+1},\theta_{2n+2}[\\
      \phi_0(X^{i,N}_t)&< a_0,\ 
      \forall t\in \cup_{n=0}^{\infty}[\theta_{2n},\theta_{2n+1}[.
    \end{split}
  \end{equation*}
  Let $ \Gamma $   be a
  $1$-dimensional Brownian motion independent of the process
  $(.,\mathbb{O}^{(N)},\mathbb{X}^{(N)})$.  We set
  \begin{equation*}
    M_t=\Gamma_t,\ \text{for}\ t\in[0,{\theta}_0[,
  \end{equation*}
  and, for all $n\geq 0$,
  \begin{align*}
    M_t&=M_{{\theta}_{2n}}+\
    \int_{{\theta}_{2n}}^{t}
	{\sum_{k=1}^{d'_0} \frac{\partial \phi_D}{\partial x_k}[\sigma_i]_{kl}(t,o^{i,N}_t,X^{i,N}_t) d[B^i_s]_{l}}
\ \text{for}\
    t\in[{\theta}_{2n},{\theta}_{2n+1}[,\\
	M_t&=M_{{\theta}_{2n+1}}+(\Gamma_t-\Gamma_{{\theta}_{2n+1}})\
	\text{for}\ t\in[{\theta}_{2n+1},{\theta}_{2n+2}[,
  \end{align*}
  Informally, $M$ is a square-integrable martingale which is parallel to the
  martingale part of $\phi_D(X^{i,N})$ when this one is near $0$ (at least strictly smaller than $a_0$), and equal
  to an independent Brownian motion when $\phi_D(X^{i,N})$ is sufficiently far from
  $0$ (at least bigger than $a_0/2$).
  By \cite[Theorem 1.9 (Knight)]{Revuz1999}, $M$ is a time
  changed Brownian motion. More precisely, there exists a
  $1$-dimensional Brownian motion $W$ such
  that, for all $t\geq 0$,
  \begin{equation*}
    M_t=W_{\langle M\rangle_t}.
  \end{equation*}
  By It\^o's formula, we have
  \begin{equation*}
    \frac{\partial \langle M\rangle_t}{\partial t}=\left\lbrace
    \begin{array}{l}
      f_i^N(t,o^{i,N}_t,X^{i,N}_t)+g_i^N(t,o^{i,N}_t,X^{i,N}_t)
        \text{ if }\exists n\geq 0\text{ such that }t\in[\theta_{2n},\theta_{2n+1}[,\\
      1\text{ if }\exists n\geq 0\text{ such that }t\in[\theta_{2n+1},\theta_{2n+2}[.
    \end{array}
    \right.
  \end{equation*}
  By Hypothesis~\ref{hypothesis:particle-dynamics}, we deduce that
  \begin{equation}
    \label{chapitre4:EqThTightnessBornM}
    c_{0}\wedge 1\leq \frac{\partial}{\partial t} \langle
    M\rangle_t\leq C_{0}\vee 1.
  \end{equation}

  \bi By the uniform bounds assumptions of Hypothesis~\ref{hypothesis:particle-dynamics}, there exists a
  positive constant $C_1>0$ such that, for all $t\in[\theta_{2n},\theta_{2n+1}[$,
  \begin{multline}
  \label{equation:bound-below-C1}
    -C_1\leq \frac{1}{2}\sum_{k,l=1}^{d_0}
      \frac{\partial^2 \phi_D}{\partial x_k\partial
      x_l}(X_t)\left[\sigma^{N}_i\left(\sigma^{N}_i\right)^*\right]_{kl}(t,o^{i,N}_t,X^{i,N}_t)\\
      +\sum_{k=1}^{d_0} \frac{\partial \phi_D}{\partial
      x_k}(X^{i,N}_t)
      \left[\eta_i^{N}\right]_k(t,o^{i,N}_t,X^{i,N}_t),
  \end{multline}
  which is the drift part of the semi-martingale $\phi_D(X^{i,N})$.
  Let $U$ be the diffusion process reflected on $0$ and $a$, defined
  by
  \begin{equation*}
    dU_t=dW_t-\frac{C_1}{c_0\wedge 1} dt+dL^{0}_t-dL^{a}_t,\ 
    U_0=0,
  \end{equation*}
  where $L^{0}$ (resp. $L^{a}$) is the local time of
  $U$ on $0$ (resp. $a$). In particular, we have
  \begin{equation*}
    dU_{\langle
      M\rangle_t}=dM_t-\frac{C_1}{c_0\wedge 1}\frac{\partial}{\partial
      t}  \langle M\rangle_t dt 
+ dL^{0}_{\langle M\rangle_t}-dL^{a}_{\langle M\rangle_t},
  \end{equation*}
  where, by the third point of Hypothesis~\ref{hypothesis:particle-dynamics} and inequalities~\eqref{chapitre4:EqThTightnessBornM} and~\eqref{equation:bound-below-C1},
  \begin{multline*}
    -\frac{C_1}{c_{0}\wedge 1}\frac{\partial}{\partial t} \langle
      M\rangle_t\leq \frac{1}{2}\sum_{k,l=1}^{d_0}
      \frac{\partial^2 \phi_D}{\partial x_k\partial
      x_l}(X_t)\left[\sigma^{N}_i\left(\sigma^{N}_i\right)^*\right]_{kl}(t,o^{i,N}_t,X^{i,N}_t)\\
      +\sum_{k=1}^{d_0} \frac{\partial \phi_D}{\partial
      x_k}(X^{i,N}_t)
      \left[\eta_i^{N}\right]_k(t,o^{i,N}_t,X^{i,N}_t),
  \end{multline*}
  which is the drift part of the semi-martingale
  $\phi_D(X^{i,N}_t)$. Informally, $U_{\langle M\rangle_t}$ evolves as
  $\phi_D(X^{i,N}_t)$ but with a stronger drift toward $0$, $U_{\langle M\rangle_t}$ is reflected on $0$
  while $\phi_D(X^{i,N}_t)$ makes positive jumps when it hits $0$ and
  $U_{\langle M\rangle_t}$ is reflected on $a$ while
  $\phi_D(X^{i,N}_t)$ can become greater than $a$. As a consequence
  (see \cite[Proposition~2.2]{Villemonais2010} for a rigorous and detailed argumentation of this fact),
  we have
  \begin{equation*}
     0\leq U_{ \langle M\rangle_t} \leq \phi_D(X^{i,N}_t),\ \forall t\in[0,T].
  \end{equation*}
  Then, for all $\gamma>0$,
  \begin{equation*}
    \left\lbrace \phi_D(X^{i,N}_t) \geq
    \gamma\right\rbrace\supset \left\lbrace U_{\langle
    M\rangle_t} \geq \gamma \right\rbrace,
  \end{equation*}
  where $\langle M\rangle_t\in[\frac{t}{C_0\vee
      1},\frac{t}{c_0\wedge 1}]$ by inequality
  \eqref{chapitre4:EqThTightnessBornM}.
  It yields that
  \begin{equation*}
    \left\lbrace \exists t\in [0,T] | \phi_D(X^{i,N}_t) \geq
    \gamma\right\rbrace\supset \left\lbrace \exists
    t\in[0,\frac{T}{c_{0}\wedge 1}] \text{ such that } U_t\geq
    \gamma,\ \right\rbrace,
  \end{equation*}
  which implies that
  \begin{equation*}
    P(S^{i,N}_{\gamma}\leq T)\geq P\left( \exists
    t\in[0,\frac{T}{c_{0}\wedge 1}] \text{ such that } U_t \geq \gamma
    \right).
  \end{equation*}
  The process $(U_t)_{t\geq 0}$ is a reflected Brownian motion with
  bounded drift, whose law doesn't depend on $i,N$. As a consequence,
  there exists $\gamma_{\epsilon}>0$ independent of $i,N$ such that
  $P\left( \exists t\in[0,\frac{T}{c_{0}\wedge 1}] \text{ such that }
  U_t \geq \gamma_{\epsilon} \right)\geq 1-\epsilon/3$. This allows us to conclude the proof of Lemma~\ref{chapitre4:lemma:bound-on-S-gamma}.
\end{proof}

\section{Uniform tightness for the conditional distribution of time-inhomogeneous diffusions}
\label{section:uniform-tightness-distribution}

In this section, we use Theorem~\ref{theorem:uniform-tightness} in order to prove the uniform
tightness of the family of conditional distributions of
time-inhomogeneous environment-dependent diffusion processes. More
precisely, let $E$ be an open subset of $\R^d$ ($d\geq 0$) and $D$ a
bounded open subset of $\R^{d'}$, with $(d'\geq 1)$. We consider the
diffusion process ${\cal Z}$ with values in $[0,+\infty[\times E\times
    D$ and denoted by ${\cal Z}_t=(t,e_t,Z_t)$ at time $t\geq 0$, which fulfils the stochastic differential system
\begin{equation}
\label{equation:stochastic-system-time-environment}
\begin{split}
  de_t&=s(t,e_t,Z_t)d\beta_t+m(t,e_t,Z_t)dt\\
  dZ_t&=\sigma(t,e_t,Z_t)dB_t+\eta(t,e_t,Z_t)dt.
\end{split}
\end{equation}
Here $(\beta,B)$ is a $d+d'$ standard Brownian motion and $s,m,\sigma,\eta$ are measurable functions.
We also assume that the process is subject to \textit{hard killing} at $\partial D$ and to \textit{soft killing} with rate $\kappa(t,e_t,Z_t)\geq0$, where $\kappa$ is a non-negative measurable function. We denote by $\tau_\d$ the killing time of $\cal Z$, defined by
$$
\tau_\d=\inf\{t\geq 0,\ {\cal Z}_t\text{ is killed at time }t\}.
$$

\bi Our first assumption ensures that the above differential system has a unique solution, and that this solution is strongly Markov.
\begin{hypothesis}
   \label{hypothesis:lipschitz}
   We assume that $s$, $m$, $\sigma$ and
   $\eta$ are continuous uniformly Lipschitz in $e,z$, uniformly in $t$.
   This means that there exists a constant $k_l>0$ such that
   \begin{multline*}
     \|s(t,e,z)-s(t,e',z')\|+\|m(t,e,z)-m(t,e',z')\|\\
     +\|\sigma(t,e,z)-\sigma(t,e',z')\|+|\eta(t,e,z)-\eta(t,e',z')|\leq k_l \left(|z-z'|+|e-e'|\right).
   \end{multline*}
\end{hypothesis}

\me
Under this hypothesis, the system~\eqref{equation:stochastic-system-time-environment} has a solution (see \cite[Theorem 3.10, Chapter 5]{Ethier1986}), which is pathwise unique and Markov up to time $\tau_{\partial}=\inf\{t\geq
0,\ Z_t\notin D\}$ (see \cite[Theorem 3.7, Chapter 5]{Ethier1986}).

\bi Our second assumption ensures that a Fleming-Viot type particle system with particles evolving as $\cal Z$ between the jumps is well defined at any time $t\geq 0$.
\begin{hypothesis}
   \label{hypothesis:ellipticity-regularity}
   We assume that
   \begin{enumerate}
   \item $\phi_D$ is of class $C^2$ on the boundary's neighbourhood $D^{a_0}$, for a given $a_0>0$ (we recall that $D^{a_0}$ is defined above Hypothesis~\ref{hypothesis:particle-dynamics}),
   \item $\kappa$ is uniformly bounded over $[0,+\infty[\times E\times D$ and $s$, $m$, $\sigma$ and $\eta$ are uniformly bounded over $[0,+\infty[\times E\times D^{a_0}$,
     \item there exist two measurable functions $f:[0,+\infty[\times E \times 
         D^{a_0}\rightarrow \mathbb{R}_+$ and $g:[0,+\infty[\times E \times 
             D^{a_0}\rightarrow \mathbb{R}$ such that $\forall
             (t,e,z)\in [0,+\infty[\times E \times D^{a_0}$,
    \begin{equation*}
      \sum_{k,l}\frac{\partial \phi_D}{\partial x_k}(z) \frac{\partial
        \phi_D}{\partial x_l}(z)
      [\sigma\sigma^*]_{kl}(t,e,z)=f(t,e,z)+g(t,e,z),
    \end{equation*}
    and such that
    \begin{enumerate}
    \item $f$ is of class $C^1$ in time and of class $C^2$ in
      environment/space, and the derivatives of $f$ are uniformly
      bounded,
    \item there exists a positive constant $k_g>0$ such that, for all
      $(t,e,z)\in[0,+\infty[ \times E \times D^{a_0}$,
      \begin{equation*}
	|g(t,e,z)|\leq k_g\phi_D(z),
      \end{equation*}
    \end{enumerate}
   \end{enumerate}
 \end{hypothesis}

\bi We are now able to state the main result of this section, which concerns the uniform tightness of the family of conditional distributions of $\cal Z$.

\begin{theorem}
\label{theorem:tightness-conditional-distribution}
Fix $t_0>0$ and assume that Hypotheses~\ref{hypothesis:lipschitz} and~\ref{hypothesis:ellipticity-regularity} hold. Then, for any $\epsilon>0$, there exists $a_{\epsilon}>0$ such that, for any initial distribution $\eta$ of ${\cal Z}$,
\begin{equation*}
\P_{\eta}\left({\cal Z}_{t}=(t,e_{t},Z_{t})\in [0,+\infty[\times E\times D^{a_{\epsilon}}|t<\tau_{\d}\right)\leq \epsilon,\ \forall t\geq t_0.
\end{equation*}
In particular, letting $(\eta_n)_{n\geq 0}$ be any sequence of initial distributions for $\cal Z$ and $(t_n)_{n\geq 0}$ be any sequence of positive times such that $\inf_{n\geq 0} t_n >0$, the family of conditional probability measures on $D$ indexed by $n\geq 0$ and defined by
\begin{equation*}
\P_{\eta_n}\left({\cal Z}_{t_n}=(t_n,e_{t_n},Z_{t_n})\in [0,+\infty[\times E\times \cdot\ |\ t_n<\tau_{\d}\right),\ \forall n\geq 0,
\end{equation*}
is uniformly tight.
\end{theorem}

\me Informally, Theorem~\ref{theorem:tightness-conditional-distribution} tells us that the conditional distribution of diffusion process doesn't degenerate to the boundary, even if its initial distribution does. We emphasize that the result still applies for random initial distributions.

\begin{proof}[Proof of Theorem~\ref{theorem:tightness-conditional-distribution}]
For any $N\geq 2$, we define the particle system $(t,\mathbb{O}^N_t,\mathbb{X}^N_t)_{t\geq 0}$ as in Section~\ref{section:definition-particle-system} with ${\cal Z}^N={\cal Z}$ and the following values of ${\cal H}^N$ and ${\cal S}^N$, which correspond to the Fleming-Viot type system introduced by Burdzy \textit{et al.}~\cite{Burdzy1996}. Assuming that the $i^{th}$ particle is killed at time $t$, the jump measures are given by
$$
{\cal H}^N(t,\mathbb{O}^N_t,\mathbb{X}^N_t)=\frac{1}{N-1}\sum_{j=1,j\neq i}^N \delta_{(o^j_t,X^j_t)}\text{ for a hard killing}
$$
and by
$$
{\cal S}^N(t,\mathbb{O}^N_t,\mathbb{X}^N_t)=\frac{1}{N-1}\sum_{j=1,j\neq i}^N \delta_{(o^j_t,X^j_t)}\text{ for a soft killing.}
$$
On the one hand, Hypothesis~\ref{hypothesis:ellipticity-regularity}, clearly implies that Hypothesis~\ref{hypothesis:particle-dynamics} is fulfilled.  On the other hand, ${\cal H}^N$ fulfils Hypothesis~\ref{hypothesis:jump-measure} with $h^N(u)=u$ and $p^N_0=1$. Thus, by Theorem~\ref{theorem:non-explosion}, we deduce that the particle system $(\cdot,\mathbb{O}^N,\mathbb{X}^N)$ is well defined at any time, for any $N\geq 2$. As a consequence, one can apply \cite[Theorem~2.1]{Villemonais2011}, which states that this implies, for any $t\geq 0$,
\begin{equation*}
\frac{1}{N}\sum_{i=1}^N\delta_{t,o^{i,N}_t,Z^{i,N}_t}(\cdot)\xrightarrow[N\rightarrow\infty]{}\P_{\eta}\left((t,e_t,Z_t)\in \cdot \times \cdot \times \cdot \right),
\end{equation*}
where the initial distributions $m^N$ of the interacting particle systems are chosen so that
\begin{equation*}
\frac{1}{N}\sum_{i=1}^N\delta_{0,o^{i,N}_0,Z^{i,N}_0}(\cdot)\xrightarrow[N\rightarrow\infty]{} \eta(\cdot).
\end{equation*}
In particular, one has, for any $t\geq 0$,
\begin{equation*}
\mu^N_t(\cdot)=\frac{1}{N}\sum_{i=1}^N\delta_{Z^{i,N}_t}(\cdot)\xrightarrow[N\rightarrow\infty]{}\P_{\eta}\left((t,e_t,Z_t)\in [0,+\infty[\times E\times \cdot \right).
\end{equation*}

\me
By Theorem~\ref{theorem:uniform-tightness}, for any $\epsilon>0$, there exists $a_{\epsilon}>0$ and $N_{\epsilon}\geq 2$ such that
$$
\E\left(\mu^N_t(D^{a_{\epsilon}})\right)\leq \epsilon,\ \forall N\geq N_{\epsilon},\ \forall t\geq t_0,
$$
independently of the initial distribution of the particle system. Thus we have
$$
\P_{\eta}\left((t,e_t,Z_t)\in [0,+\infty[\times E\times D^{a_{\epsilon}} \right)\leq \epsilon,
$$ 
for any initial distribution $\eta$. These concludes the proof of the first part of Theorem~\ref{theorem:tightness-conditional-distribution} and  immediately implies its second part.

\end{proof}

\paragraph{Acknowledgements} This work has been partly written during mly PhD thesis and benefited from the support of the "Chaire Mod\'elisation Math\'ematique et Biodiversit\'e of Veolia Environnement-\'Ecole Polytechnique-Museum National d'Histoire Naturelle-Fondation X". I am extremely grateful to my PhD advisor Sylvie M\'el\'eard for her support and to Steven N. Evans, reviewer of my PhD thesis, for his comments on the first version of this paper.

\end{document}